\documentclass[amscd,amssymb,verbatim,11pt]{amsart}

\usepackage[usenames,dvipsnames,svgnames,table]{xcolor}
\usepackage[pagebackref,linktocpage=true,colorlinks=true,linkcolor=Blue,citecolor=BrickRed,urlcolor=RoyalBlue]{hyperref}
\usepackage[alphabetic,msc-links,abbrev]{amsrefs} 

\usepackage{color}
\usepackage{graphics}
\usepackage[dvips]{graphicx}
\usepackage{wrapfig}
\usepackage{fancyhdr}
\usepackage{latexsym}
\usepackage{mathrsfs}
\usepackage{amsmath}
\usepackage{amssymb}
\usepackage[all]{xy}
\usepackage{fancyhdr}
\usepackage{fancybox}
\usepackage{rotating}

\author[Aram L. Karakhanyan]{Aram L. Karakhanyan}
\title[Monotonicity Formula]{A monotonicity formula for a classical free boundary problem}
\address{School of Mathematics, The University of Edinburgh, EH9 3FD, UK}
\email{aram6k@gmail.com}
\thanks{$2000$ {\it Mathematics Subject Classification.\/} Primary
 35R35, 35K05, 35K55.}
\date{}

\theoremstyle{plain}
\newtheorem{theorem}{Theorem}[section]
\newtheorem{lemma}[theorem]{Lemma}

\newtheorem{prop}[theorem]{Proposition}
\newtheorem{corollary}[theorem]{Corollary}
\theoremstyle{remark}
\newtheorem{remark}[theorem]{Remark}
\theoremstyle{definition}

\renewcommand\epsilon\varepsilon 
\renewcommand\phi\varphi 
\renewcommand\div{\operatorname{div}} 
\newcommand\R{\mathbb{R}} 
\numberwithin{equation}{section}

\newcommand\I[1]{\chi_{\{#1\}}}  
\newcommand\fb[1]{\partial{\{#1>0\}}}  

\renewcommand\H{\mathscr H}

\usepackage{esint}

\begin{document}

\begin{abstract}
We construct a  monotonicity formula for the free boundary problem of the form 
$\Delta u=\mu$, where 
$\mu$ is  a Radon measure.  It implies that the blow up limits of solutions are 
homogenous functions of degree one. The first formula is new even for classical Laplace operator.

Our method of proof uses a careful application of the 
strong maximum principle. 
\end{abstract}

\maketitle

%
%


\section{Introduction}
In this note we construct a monotonicity formula for the local  
minimizers of the functional
\begin{eqnarray}\nonumber
 J(u)=\int_{\Omega}|\nabla u|^2
+
\lambda \I {u>0},
\end{eqnarray}
where $\Omega\subset \R^n$ and $\chi _D$ is the characteristic function of a set $D$.
Our main result is

\begin{theorem}\label{theorem-1}
Let $u$ be a local minimizer of $J$, and $0\in \fb u$, then 
\begin{equation}
K(r)=\frac1{|B_r|}\int_{B_r}\I {u>0}
\end{equation}
is a monotone non-decreasing function of $R$. Furthermore, $K$ is constant if
and only if the set $\{u>0\}$ is a cone and $u$ is a homogeneous function of degree one.
\end{theorem}
There are a number of monotonicity formulae for this problem, see \cite{APh}, \cite{ACF}, \cite{Spruck}, \cite{W-cpde}.
However, I believe that the monotonicity of $K(r)$ has not been proved in the literature even for the classical Alt-Caffarelli problem \cite{AC}.  
In view of the Allard monotonicity formula \cite{Simon}, $K(R)$
is the natural density function that enjoys monotonicity. 
The proof of Theorem \ref{theorem-1} is hinging on a
local analysis of the free boundary structure by employing the 
strong maximum principle.

%
%
\section{The maximum principle}
We use a slightly general set up. Let 
$F(t):\R_+\mapsto \R_+, F(t)\in C^{2, 1}[0, \infty]$ satisfies the following conditions:
\begin{itemize}
 \item[$(\bf F_1)$] \quad $F(0)=0,\qquad  c_0\leq F'(t)\leq C_0, $ 
\item[$(\bf F_2)$] \quad $\displaystyle 0\leq F''(t)\leq \frac{C_0}{1+t}, $ 
\end{itemize}
for some positive constants $c_0, C_0$, see  \cite{ACF-Quasi} page 2. 
Let 

\begin{equation}
J_F=\int_\Omega F(|\nabla u|^2)+\lambda \I{u>0}
\end{equation}


Note that $\div(F'(|\nabla u|^2)\nabla u)$
is a measure supported on the boundary of the set 
$\{u>0\}$. 
If $u$ is a local minimizer of $J_F$, then 
$a_{ij}u_{ij}=0$ in $\Omega^+=\{u>0\}\cap \Omega$, where $a_{ij}=F'(|\nabla u|^2)\delta_{ij}+2F''(|\nabla u|^2)u_iu_j$, see \cite{ACF-Quasi}. 
Recall the free boundary condition is of the form
\begin{equation}
2|\nabla u|^2F'(|\nabla u|^2)-F(|\nabla u|^2)=\lambda, 
\end{equation}
see \cite{ACF-Quasi}. 
Consequently $|\nabla u|=\lambda^*$ on the free boundary, 
for some constant $\lambda^*$ determined from the above 
condition. Throughout this paper as assume that $\lambda^*=1$. 
\begin{prop}\label{prop-strong-max}
Let $u\in C^3(\Omega)$ be a solution to $a_{ij}u_{ij}=0$ in $\Omega$, where 
$a_{ij}=F'(|\nabla u|^2)\delta_{ij}+2F''(|\nabla u|^2)u_iu_j$. 
Then the function $w(x)=\nabla u(x)\cdot x -u(x)$ satisfies 
\begin{equation}
a_{ij}w_{ij}=b\cdot \nabla w, 
\end{equation}
where 
\begin{equation}
b=-2\left\{\left[F''(|\nabla u|^2)\Delta u +2 F'''(|\nabla u|^2)\nabla u D^2 u\nabla u\right]\nabla u+2F''(|\nabla u|^2)\nabla uD^2 u\right\}.
\end{equation}
\end{prop}

\begin{remark}
If $F$ is a linear function, i.e. $a_{ij}u_{ij}=\Delta u$, then $b=0$.
\end{remark}

\begin{proof}
The derivatives of $w$ can be readily computed as follows
\begin{align}
w_i=u_{mi}x_m, \\
w_{ij}=u_{mij}x_m+u_{ij}.
\end{align}
Contracting the Hessian $D^2w$ with the coefficient matrix $a_{ij}$ and 
using the equation $a_{ij}u_{ij}=0$ we get 
\begin{align*}
a_{ij}w_{ij}&= a_{ij}u_{mij}x_m+a_{ij} u_{ij}=\\
&= 
a_{ij}u_{mij}x_m\\
&=
-\partial_{x_m}(a_{ij})u_{ij}x_m.
\end{align*}
Using the definition of $a_{ij}$ we have  
\begin{align*}
\partial_{x_m}(a_{ij})
&=
\partial_{x_m}(F'\delta_{ij}+2F''u_iu_j)\\
&=
2F''\nabla u\nabla u_m\delta_{ij}+4F'''\nabla u\nabla u_m u_iu_j
+
2F''(u_{im}u_j+u_iu_{mj}).
\end{align*}
Consequently, 
\begin{align*}
-\frac12 \partial_{x_m}(a_{ij})u_{ij}x_m
&=
-F''\nabla u(\nabla u_mx_m)\Delta u-2F'''\nabla uD^2 u\nabla u (\nabla u_mx_m)\nabla u\\
&
-2F''\nabla uD^2 u (\nabla u_m x_m)\\
&=
-\Big\{\left[F''(|\nabla u|^2)\Delta u +2 F'''(|\nabla u|^2)\nabla u D^2 u\nabla u\right]\nabla u\\
&+2F''(|\nabla u|^2)\nabla uD^2 u\Big\}\cdot \nabla w\\
&=
\frac b2\cdot \nabla w. 
\end{align*}
We use the notation 
\begin{equation}\label{b}
b:=-\Big\{\left[F''(|\nabla u|^2)\Delta u +2 F'''(|\nabla u|^2)\nabla u D^2 u\nabla u\right]\nabla u+2F''(|\nabla u|^2)\nabla uD^2 u\Big\}
\end{equation}

Therefore, 
\begin{equation}\label{eq-b-der}
\partial_{x_m}(a_{ij})u_{ij}x_m=-b\cdot \nabla w=-b\cdot (\nabla u_m x_m),
\end{equation}
and the result follows.
\end{proof}
As a direct consequence of the strong maximum principle we have 
\begin{corollary}
If $0\in \fb u$, and $w$ achieves a local maximum or minimum in $\Omega^+$, then 
$w=0$ throughout $\Omega^+$, and $\Omega^+$ is a cone.
\end{corollary}
\begin{proof}
If $w$ achieves a local maximum or minimum in $\Omega^+$, then 
by the strong maximum principle $w$ must be constant throughout $\Omega^+$. Since $w(0)=0$, it follows that 
$w=0$ throughout $\Omega^+$, and hence $u$ is homogeneous of degree one, i.e. 
$u(tx)=tu(x),$ for any  $t>0$
and $\Omega^+$ is a cone. 
\end{proof}

\begin{prop}\label{prop-main}
Let $u$ be a local minimizer of $J$ and $0\in \fb u$, and assume that 
$b\cdot x\le 0$ for any local minimizer having zero on its free boundary. Here $b$ is defined by \eqref{b}.
Then 
for any $x_0\in \fb u$ we have 
\begin{equation}\label{eq-ell}
\limsup_{
\substack{
x\to x_0\\
x\in \{u>0\}}
}\left(\nabla u \cdot \frac x{|x|}-\frac {u(x)}{|x|}\right)\ge 0.
\end{equation}

\end{prop}
\begin{proof}
Suppose \eqref{eq-ell} fails, then 
for some $x_0\in \fb u$ we must have that 
\begin{equation}\label{eq-ell-fail}
\limsup_{
\substack{
x\to x_0\\
x\in \{u>0\}}
}\left(\nabla u(x) \cdot \frac x{|x|}-\frac {u(x)}{|x|}\right)=-\ell^2<0.
\end{equation}
Let $x_k\in \{u>0\}, x_k\to x_0$,  and 
\begin{equation}\label{}
\lim_{ x_k\to x_0
}\left(\nabla u(x_k) \cdot \frac {x_k}{|x_k|}-\frac {u(x_k)}{|x_k|}\right)=-\ell^2<0.
\end{equation}
Let $y_k\in \fb u$ be the closest point to 
$x_k$ on the free boundary. Denote $\rho_k=|x_k-y_k|$. 
We consider three cases:
\begin{itemize}
\item[Case 1:] $\frac{|y_k|}{\rho_k}\to 0$, 
\item[Case 2:]  $\frac{|y_k|}{\rho_k}\to a\not =0$,
\item[Case 3:]  $\frac{|y_k|}{\rho_k}\to \infty$.
\end{itemize}

\subsection{Case 1:}
We introduce the function 
$u_k(y)=\frac{u(x_k+\rho_k y)}{\rho_k}$. Since $\rho_k=|x_k-y_k|\le |x_k-x_0|\to 0$, it follows that 
we can employ a customary compactness argument and show that for a subsequence 
$u_k(x)\to \bar u(x)$, and $\bar u(x)$ is a local minimizer of $J$, see  \cite{ACF-Quasi}.  
Moreover, since $x=y_k+\rho_k y\to x_0$ for $y\in B_2$, and 
\[
\frac{y_k+\rho_ky}{|y_k+\rho_k y|}=\frac{\frac{y_k}{\rho_k}+y}{\left|\frac{y_k}{\rho_k}+y\right|}\to \frac y{|y|}, 
\quad \frac{\rho_k}{|y_k+\rho_k y|}=\frac 1{\left|\frac{y_k}{\rho_k}+y\right|}\to \frac1{|y|}, 
\]
we get from \eqref{eq-ell-fail} that
\begin{equation}\label{eq-ell-fail-00}
\nabla \bar u(y) \cdot \frac y{|y|}-\frac {\bar u(y)}{|y|}\le -\ell^2.
\end{equation}
Moreover, by construction, there is a point $y_0\in \{u_0>0\}, |y_0|=1$, such that 
$B_1(y_0)\in \{\bar u>0\}$, $0\in \partial B_1(y_0)\cap \partial \{\bar u>0\}$ and 
\begin{equation}\label{eq-ell-fail-01}
\nabla \bar u(y_0) \cdot \frac{ y_0}{|y_0|}-\frac {\bar u(y_0)}{|y_0|}=-\ell^2.
\end{equation}
Introduce, $v(y)=\frac {\bar u(y)}{|y|}-\nabla \bar u(y) \cdot \frac{ y}{|y|}$. 
Then it follows from \eqref{eq-ell-fail-00} and \eqref{eq-ell-fail-01}, 
that $v(y)\ge \ell^2$ in $B_1(y_0)\subset \{\bar u>0\}$ and $v(y_0)=\ell^2.$
Therefore, $v$ has a local minimum at $y_0$. We next show that 
$v(y)$ is constant in the connected component of $\{\bar u>0\}$ containing $y_0$. 
To see this we employ the strong maximum principle.
First observe that 
\begin{align*}
v_i
&=
\frac{\bar u_i}{|y|}-\frac{\bar u y_i}{|y|^3}-\frac{\bar u_{mi}y_m+\bar u_i}{|y|}+\frac{\bar u_my_my_i}{|y|^3}\\
&=
-\frac{\bar uy_i}{|y|^3}-\frac{\bar u_{mi}y_m}{|y|}+\frac{\bar u_my_my_i}{|y|^3}\\
&=
-\frac{ vy_i}{|y|^2}-\frac{\bar u_{mi}y_m}{|y|}.
\end{align*}
Next, we compute the Hessian
\begin{align*}
v_{ij}
&=
-\frac{\bar u_{mij}y_m}{|y|}-\frac{\bar u_{ij}}{|y|}+\frac{\bar u_{mi}y_my_j}{|y|^3}\\
&
+\frac{\bar u_{mj}y_my_i+\bar u_m\delta_{mj}y_i+\bar u_my_m\delta_{ij}}{|y|^3}-3\frac{\bar u_my_my_iy_j}{|y|^5}\\
&
-\frac{\bar u_jy_i}{|y|^3}-\frac{\bar u\delta_{ij}}{|y|^3}+3\frac{\bar uy_iy_j}{|y|^5}\\
&=
-\frac{\bar u_{mij}y_m}{|y|}-\frac{\bar u_{ij}}{|y|}+\frac{\bar u_{mi}y_my_j+\bar u_{mj}y_my_i}{|y|^3}-\frac{\delta_{ij}}{|y|^3}(\bar u-\bar u_my_m)\\
&+
3\frac{y_iy_j}{|y|^5}(\bar u-\bar u_my_m). 
\end{align*}
 After contracting the Hessian with the matrix $a_{ij}$ we obtain
 \begin{align*}
 a_{ij}v_{ij}
 &=
 -\frac1{|y|}a_{ij}\bar u_{ijm}y_m-\frac1{|y|}a_{ij}\bar u_{ij}+\frac2{|y|^3}a_{ij}y_j(\bar u_{mi}y_m)\\
 &+
 \left(3\frac{a_{ij}y_iy_j}{|y|^5}-
 \frac{\textrm{Trace}a_{ij}}{|y|^3}\right)(\bar u-\nabla \bar u \cdot y)\\
 &=
 \frac1{|y|}\partial_{y_m}(a_{ij})\bar u_{ij}y_m+\frac2{|y|^3}a_{ij}y_j(\bar u_{mi}y_m)\\
 &+
 \left(3\frac{a_{ij}y_iy_j}{|y|^5}-
 \frac{\textrm{Trace}a_{ij}}{|y|^3}\right)(\bar u-\nabla \bar u \cdot y)\\
 &=
 -b\cdot \frac{\nabla \bar u_my_m}{|y|}+\frac2{|y|^3}a_{ij}y_j(\bar u_{mi}y_m)\\
 &+
 \frac{v}{|y|^2}\left(3\frac{a_{ij}y_iy_j}{|y|^2}-
 \textrm{Trace}a_{ij}\right),
 \end{align*}
where to get the last line we used \eqref{eq-b-der}.
From the computation of the first derivatives we see that 
$\nabla v= -\frac y{|y|^3}v-\frac{\nabla \bar u_my_m}{|y|}$, hence $a_{ij}v_{ij}$ can be 
simplified as follows 
\begin{align*}
a_{ij}v_{ij}
&=
b\cdot \left[ \nabla v+\frac y{|y|^3}v \right]-\frac2{|y|^2}a_{ij}y_j \left[ v_i+\frac {y_i}{|y|^2}v \right]\\
&+
 \frac{v}{|y|^2}\left(3\frac{a_{ij}y_iy_j}{|y|^2}-
 \textrm{Trace}a_{ij}\right)\\
 &=
 \tilde b\cdot \nabla v+c v, 
\end{align*}
where 
\[
\tilde b^k=b^k-\frac{2a_{kj}y_j}{|y|^2}, \quad c=\frac{b\cdot y}{|y|^3}-\frac{2a_{ij}y_iy_j}{|y|^4}
+
 \frac{1}{|y|^2}\left(3\frac{a_{ij}y_iy_j}{|y|^2}-
 \textrm{Trace}a_{ij}\right).
\]
Decomposing $c=c^+-c^-$ into positive and negative parts, and noting that 
$v\ge \ell^2>0$  it follows that 
\[
a_{ij}v_{ij}= \tilde b\cdot \nabla v+c^+ v-c^-v\le \tilde b\cdot \nabla v+c^+ v, 
\]
or equivalently 
\[
a_{ij}v_{ij}-\tilde b\cdot \nabla v-c^+ v\le 0.
\]
Since by assumption $b\cdot y\le 0$, it follows that $c\le \frac{1}{|y|^2}\left(\frac{a_{ij}y_iy_j}{|y|^2}-
 \textrm{Trace}a_{ij}\right)\le 0$. Consequently, $c^+=0$. Note that for the Laplace operator $b\equiv0$. 

Hence from the strong maximum principle we conclude that 
$v(y)=\ell^2$ for any $y\in \{\bar u>0\}$. 

In order to finish the proof for the Case 1, we will show that 
$\ell^2\not =0$ contradicts the fact that the mean curvature of the 
the smooth portions of free boundary is positive, see Appendix.
 On the smooth portion of the free boundary the 
 condition 
 $v=\ell^2$ means that $ \nabla \bar u\cdot y=-\ell^2|y|.$
 Note that $\nabla \bar u$ is the unit normal pointing into 
 $\{\bar u>0\}$. Let $y(s)$ be an arc-length parametrization of a 
 planar curve in  the free boundary near some regular point, say $z_0\not =0$,  and $z_0=y(0)$.
 Then $N\cdot y(s)=-\ell^2|y(s)|$ , where $N(s)=\nabla \bar u(y(s))$ is the unit normal 
 on the curve pointing into $\{\bar u>0\}$. Differentiating in arc-length parameter $s$ yields 
 \[
 \dot N\cdot y+N\cdot \dot y=-\ell^2\frac{\dot y\cdot y}{|y|}.
 \]  
 Since $\dot y=\tau$ is the unit tangent we get that $N\cdot \dot y=0$, and from the 
 Frenet equations $\dot N=-k\tau, \dot \tau =k N$, we can rewrite the last equation in the 
 following equivalent form 
 \[
 -k\tau\cdot y=-\ell^2\frac{\tau \cdot y}{|y|}.
 \]
 Cancelling $\tau\cdot y$ and choosing $\dot y(0)$ as a principal direction 
 we see that 
 \begin{equation}
 \sum_{i=1}^{n-1}-k_i=-(n-1)\frac{\ell^2}{|y|}.
 \end{equation}
Differentiating $\bar u(y(s))=0$ twice in $s$ we obtain 
$\bar u_{\tau \tau }+\nabla \bar u\cdot \dot \tau =0$. In the 
principal coordinate system coordinates, at $z_0$,  the 
equation $a_{ij}\bar u_{ij}=0$ yields $F'(1)\sum_{i=1}^{n-1}\bar u_{\tau_i\tau_i}+(F'(1)+2F''(1))\bar u_{NN}=0.$
From here we infer that at $z_0$
\begin{equation}
-\left (1+2\frac{F''(1)}{F'(1)}\right )\bar u_{NN}=\sum_{i=1}^{n-1}\bar u_{\tau_i\tau_i}
=
-\sum_{i=1}^{n-1}N\cdot \dot \tau_i=-\sum_{i=1}^{n-1}k_i=-(n-1)\frac{\ell^2}{|y|}.
\end{equation}
Since $h(y)=\nabla \bar u(y)\cdot {N(z_0)}$ has its maximum achieved at the 
boundary point $z_0$, then from Hopf's lemma it follows that 
$\partial _{-N_0}h>0$ at $z_0$, or equivalently $\bar u_{N_0N_0}<0$, where $N_0=N(z_0)$.  
This yields
\[
-(n-1)\frac{\ell^2}{|y|}>0,
\]
which is a contradiction. 

\subsection{Case 2:} As for the Case 2, we first observe that 
the blow up limit $\bar u$ of $u_k(y)=\frac{u(y_k+ y \rho_k)}{\rho_k}, y\in B_2$,
satisfies the inequality 
\[
 \nabla  \bar u(y)\cdot\frac{E_0+y/a}{|E_0+y/a|}-\bar u(y)\frac{1/a}{|E_0+y/a|}\le -\ell^2, 
\]
with equality holding at some interior point $y_0$ such that $B_1(y_0)\subset \{\bar u>0\}$.
Observe that 
\begin{align}
\frac{x_k+y\rho_k}{|x_k+y\rho_k|}
= \frac{\frac{x_k}{|x_k|}+y\frac{\rho_k}{|y_k|}}{\left|\frac{x_k}{|x_k|}+y\frac{\rho_k}{|y_k|}\right|}
\to \frac{E_0+y/a}{|E_0+y/a|}, \\ 
\frac{\rho_k}{|x_k+y \rho_k|}
= \frac{\frac{\rho_k}{|y_k|}}{\left|\frac{x_k}{|x_k|}+y\frac{\rho_k}{|y_k|}\right|}.
\to \frac{1/a}{|E_0+y/a|},
\end{align}
for some unit vector $E_0$.

Setting $U(z)=u(a(z-E_0)), z=E_0+y/a$, we can 
reduce this case  to to the previous one, since 
$\nabla U\cdot \frac z{|z|}-\frac{U(z)}{|z|}\le -a\ell^2$, 
with equality holding at some interior point of $\{U>0\}$, and hence 
preceding proof applies to this case too. 
After applying the strong maximum principle one can see that 
\begin{equation}
 \nabla  \bar u(y)\cdot\frac{E_0+y/a}{|E_0+y/a|}-\bar u(y)\frac{1/a}{|E_0+y/a|}=-\ell^2, 
\end{equation}
in $\{\bar u>0\}$. From this point on, we can either proceed like in the 
previous case, or, alternatively, one may  consider the  blow out limit, i.e. 
taking $y=z_0+Rz$, where $z_0\in \fb {\bar u}$ and let $R\to \infty.$
It is obvious that the scaled functions $\bar u_R(z)=\bar u(z_0+Rz)/R$ are local minimizer, 
hence applying a customary  compactness argument, one readily verifies that 
for any sequence $R_k\to \infty$, there is a  subsequence still denoted by 
$R_k$, so that the functions $\bar u_k(z)=\bar u(z_0+R_kz)/R_k$ converge 
locally uniformly to a  local minimizer $\bar u ^\infty$, such that 
\begin{equation}
 \nabla  \bar u^\infty(z)\cdot\frac{z}{|z|}-\bar u^\infty(z)\frac{1}{|z|}=-\ell^2. 
\end{equation}
This, as we have seen in the proof of the previous case, yields that $\ell=0$, and this finishes the 
argument for Case 2.

\subsection{Case 3:} For the remaining case, $\frac {|y_k|}{\rho_k}\to \infty$, we first assume that $x_0\not =0.$
Let $\sigma_k=|x_0-y_k|$ and take $x=x_0+\sigma_k y, y\in B_2$.
The scaled function $u_k(y)=u(x_0+\sigma_k y)/\sigma_k$, for a suitable
subsequence, converge to a local minimizer $\bar u$, such that 
for $y\in B_2\cap \{ \bar u>0\}$, 
\[
\nabla \bar u(y)\cdot x_0\leq -\ell^2|x_0|.
\]

Hence 
for $e_0=-\frac{x_0}{|x_0|}$, we have 
$\nabla \bar u(y) \cdot e_0\ge {\ell^2}$.
We infer from here that the minimizer $\bar u$ is strictly monotone near 
the free boundary point $y_0\lim_{k\to \infty}\frac {y_k}{\sigma_k}$, for a suitable subsequence. 
Applying the flatness implies $C^{1, \alpha}$
regularity theory \cite{Feldman}, we see that $\fb {\bar u}$ is smooth near $y_0$.
Consequently, from the convergence of $\fb {u_k}$ to $\fb {\bar u}$ in the Hausdorff distance \cite{ACF-Quasi},
it follows that the free boundaries $\fb {u_k}$ are uniformly smooth near $y_k\in \fb {u_k}$. 
Therefore, by construction, there is a point $y_0\in \partial B_1\cap \fb {\bar u}$ such that 
\begin{equation}\label{}
\lim_{
\substack{
y\to y_0\\
y\in \{\bar u>0\}}
}\left(\nabla \bar u(y) \cdot x_0\right)=-\ell^2|x_0|<0.
\end{equation}
It follows that $y_0$ is a maximum point for $h(y)=\nabla \bar u(y)\cdot x_0$.
Letting $N=\nabla \bar u(y_0)$, the inner unit normal pointing into $\{\bar u>0\}$, we see 
from the Hopf lemma that $\partial_{-N}h(y_0)>0$. Decomposing $x_0=\alpha N+\beta T$, where 
$T\perp N$, we see that 
\[
0< \partial_{-N}h(y_0)=\alpha (-\bar u_{NN}(y_0))-\beta \bar u_{NT}(y_0).
\]
The last term on the right hand side vanishes at $y_0$. Indeed, let $y(s)$
be an arc-length parametrized curve such that $y(0)=y_0$, and 
$\dot y(0)=T$. Differentiating the free boundary condition $|\nabla \bar u(y(s))|^2=1$ in $s$, we 
obtain that $\bar u_{NT}(y_0)=0$. 
Furthermore, from $\nabla \bar u(y_0)\cdot x_0<0$ we get that $\alpha<0$, and from the 
mean curvature equation $-\bar u_{NN}(y_0)>0$ (see the Appendix), it follows that 
$\alpha (-\bar u_{NN}(y_0))<0$. This is in  contradiction with 
$\partial_{-N}h(y_0)>0$.

Finally, assume that $x_0=0$. In this case, we scale $u$ by $\sigma_k=|y_k|\to 0$, then, as before,  
we can employ the same procedure 
to construct a blow up limit $\bar u(y)$, such that $\bar u$
is a local minimizer and 
 
\begin{equation}\label{eq-mon-K}
 \nabla  \bar u(y)\cdot\frac{y}{|y|}-\bar u(y)\frac{1}{|y|}\le-\ell^2, \quad y\in B_1\cap \{\bar u>0\}. 
\end{equation}
On the regular free boundary points $ \nabla \bar u(y)\cdot{y}\le -\ell^2{|y|}$.
Introduce the density function 
\begin{equation}
K(r)=\frac1{|B_r|}\int_{B_r}\I {u>0}.
\end{equation}
Since $\Gamma:=\fb {\bar u}$ has locally finite perimeter and the reduced
boundary is a dense open subset, it follows from the divergence theorem that 
\begin{align*}
\int_{\Gamma\cap B_R}\nabla \bar u(y)\cdot y
&=
-\left[\int_{\Gamma\cap B_R}-\nabla \bar u\cdot x+\int_{\{\bar u>0\}\cap \partial B_R}\frac x{|x|}\cdot x\right]
+
\int_{\{\bar u>0\}\cap \partial B_R}\frac x{|x|}\cdot x\\
&=
-\int_{B_R\cap \{\bar u>0\}}\div x
+
\int_{\{\bar u>0\}\cap \partial B_R}\frac x{|x|}\cdot x\\
&=
-n\int_{B_R}\I{\bar u>0}+R \int_{\partial B_R}\I{\bar u>0}.
\end{align*}
  Consequently, we infer the formula
  \begin{align}\label{eq-K-deriv}
  \frac{ dK}{dR}=\frac1{R^{n+1}}\int_{\Gamma\cap B_R}\nabla \bar u(y)\cdot y
  \end{align}
  which, in view of $ \nabla \bar u(y)\cdot{y}\le -\ell^2{|y|}$, satisfied $\H^{n-1}$ a.e. on 
  the free boundary, states that $K(R)$ is a non increasing bounded function.
  For any sequence of positive numbers $R_k\to 0$ and two positive numbers $a<b$ we observe  that 
  the following formula is true
  \begin{equation}
  \int_a^b\frac{dR}{R^{n+1}}\int_{\Gamma_k\cap B_R}\nabla \bar u_k(x)\cdot x=K(bR_k)-K(aR_k)\to 0
  \end{equation}
where $\bar u_k(y)=\bar u(R_k y)/R_k$ and $\Gamma_k=\fb {u_k}$. 
Using a standard compactness argument we can conclude that there is a 
limit $\bar u_0$, such that 
 \begin{equation}
  \int_a^b\frac{dR}{R^{n+1}}\int_{\Gamma_0\cap B_R}\nabla \bar u_0(x)\cdot x=0
  \end{equation}
  Moreover, \eqref{eq-mon-K}
  translates to $\bar u_0$, and we get that on $\Gamma_0=\fb{\bar u_0}$ 
  the inequality  $ \nabla \bar u_0(y)\cdot{y}\le -\ell^2{|y|}$ holds $\H^{n-1}$ a.e., thus 
  \begin{equation}
 0\le  
 \int_a^b\frac{dR}{R^{n+1}}\int_{\Gamma_0\cap B_R}\ell^2|y|
 \le 
 -\int_a^b\frac{dR}{R^{n+1}}\int_{\Gamma_0\cap B_R}\nabla \bar u_0(x)\cdot x=0.
  \end{equation}
  This, yields $\ell=0$ and we finish the proof.
\end{proof}

%
%

\section{Proof of Theorem \ref{theorem-1}}
Recalling \eqref{eq-K-deriv} we see that 
\begin{align}
  \frac{ dK}{dr}=\frac1{r^{n+1}}\int_{\Gamma\cap B_r}\nabla u(x)\cdot x.
  \end{align}
It follows from Proposition \ref{prop-main} that $K'(r)\ge 0$ because on the 
regular portion of the free boundary $\nabla u\cdot x\ge 0$, and hence 
$\H^{n-1}$ a.e. on $\fb u$. 
Repeating the argument in the proof of Case 3, we see that 
for any blow up of $u$ at 0 we must have 
$\nabla u_0(x)\cdot x=0$,  $\H^{n-1}$ a.e. on $\fb {u_0}$.
Now the homogeneity of $u_0$ follows from Lemma \ref{lem-app-2}.

%
%
\section{Appendix}
It is well known that the mean curvature of the free boundary is a 
positive multiple of $-\partial_{NN}^2 u$ at the regular free boundary points 
\cite{K}.
Combining this with Hopf's lemma we can prove that some 
local properties of the regular free boundary points 
instantaneously become global. 
\begin{lemma}
Let $u$ be a blow up limit of a local minimizer, and $x_0\in \fb u$ is a regular point 
of the free boundary. Then either $u$ is a linear function in $\{u>0\}$ or 
$-\partial_{NN}^2 u(x_0)>0$, where $N=\nabla u(x_0)$.
\end{lemma}
\begin{proof}
Since $u$ is a blow up limit then it follows from \cite{ACF-Quasi}
that $|\nabla u|\le1$, with equality on the regular free boundary points. 
Let $N=\nabla u(x_0)$ be the inner unit normal point to 
$\Omega ^+$, then $h(x)=\nabla u(x)\cdot N$ achieves its maximum at 
$x_0$. From 
Hopf's lemma it follows that $\partial_{-N}h(x_0)>0$, or equivalently, 
$-\partial ^2_{NN}u(x_0)>0$ unless $h$ is constant. The latter case 
implies that near $x_0$ the free boundary is a piece of a hyperplane, 
and consequently there is a linear function $L(x)$ such that 
$a_{ij}(u-L)_{ij}=0$ in $\{u>0\}$ near the planar piece of the free boundary. 
Applying Hopf's lemma to $u-L$ one more time the result follows. 
\end{proof}

\begin{lemma}\label{lem-app-2}
Let $u$ be as in previous lemma. Moreover, let us assume that 
$\nabla u (x)\cdot x=0$ for $x\in \fb u$ in some neighborhood of a regular free boundary point $x_0$. 
Then  $u$ is a homogeneous function of degree one on the 
component of the positivity set containing  $x_0$ on its boundary.
\end{lemma}
\begin{proof}
 Let $N=\nabla u(z)$, $z$
close to $x_0$, then 
$\partial_N w(z)=ND^2(z)z=0$, where $w(x)=\nabla u(x)\cdot x-u(x)$. Indeed, 
since $z\cdot N=0$, it follows that 
$z$ is a tangential direction at $z$. Differentiating the 
free boundary condition $|\nabla u|^2=1$ in tangential 
directions, we see that $u_{NT}(z)=0$ for any tangential 
direction $T$, 
and the claim follows.  This shows that $w(z)=|\nabla w(z)|=0$ for every 
$z$ close to $x_0$. By Proposition \ref{prop-strong-max} $w(x)=\nabla u(x)\cdot x-u(x)$ 
solves the uniformly elliptic equation $a_{ij} w_{ij}=b\cdot \nabla w$. 
Consequently, $w=0$ as claimed. 
\end{proof}

%
%
\section*{\footnotesize Acknowledgements}
The author was partially supported by EPSRC grant EP/S03157X/1 Mean curvature
measure of free boundary.
\footnotesize{}



\begin{bibdiv}
\begin{biblist}

\bib{AC}{article}{
   author={Alt, H. W.},
   author={Caffarelli, L. A.},
   title={Existence and regularity for a minimum problem with free boundary},
   journal={J. Reine Angew. Math.},
   volume={325},
   date={1981},
   pages={105--144},
   issn={0075-4102},
   review={\MR{618549}},
}
\bib{APh}{article}{
   author={Alt, H. W.},
   author={Phillips, D.},
   title={A free boundary problem for semilinear elliptic equations},
   journal={J. Reine Angew. Math.},
   volume={368},
   date={1986},
   pages={63--107},
   issn={0075-4102},
   review={\MR{850615}},
}
\bib{ACF}{article}{
   author={Alt, Hans Wilhelm},
   author={Caffarelli, Luis A.},
   author={Friedman, Avner},
   title={Variational problems with two phases and their free boundaries},
   journal={Trans. Amer. Math. Soc.},
   volume={282},
   date={1984},
   number={2},
   pages={431--461},
   issn={0002-9947},
   review={\MR{732100}},
   doi={10.2307/1999245},
}

\bib{ACF-Quasi}{article}{
   author={Alt, Hans Wilhelm},
   author={Caffarelli, Luis A.},
   author={Friedman, Avner},
   title={A free boundary problem for quasilinear elliptic equations},
   journal={Ann. Scuola Norm. Sup. Pisa Cl. Sci. (4)},
   volume={11},
   date={1984},
   number={1},
   pages={1--44},
   issn={0391-173X},
   review={\MR{752578}},
}


\bib{Feldman}{article}{
   author={Feldman, Mikhail},
   title={Regularity of Lipschitz free boundaries in two-phase problems for
   fully nonlinear elliptic equations},
   journal={Indiana Univ. Math. J.},
   volume={50},
   date={2001},
   number={3},
   pages={1171--1200},
   issn={0022-2518},
   review={\MR{1871352}},
   doi={10.1512/iumj.2001.50.1921},
}



\bib{K}{article}{
   author={Karakhanyan, Aram},
   title={Full and partial regularity for a class of nonlinear free boundary
   problems},
   journal={Ann. Inst. H. Poincar\'{e} C Anal. Non Lin\'{e}aire},
   volume={38},
   date={2021},
   number={4},
   pages={981--999},
   issn={0294-1449},
   review={\MR{4266232}},
   doi={10.1016/j.anihpc.2020.09.008},
}


\bib{Simon}{book}{
   author={Simon, Leon},
   title={Lectures on geometric measure theory},
   series={Proceedings of the Centre for Mathematical Analysis, Australian
   National University},
   volume={3},
   publisher={Australian National University, Centre for Mathematical
   Analysis, Canberra},
   date={1983},
   pages={vii+272},
   isbn={0-86784-429-9},
   review={\MR{756417}},
}

\bib{Spruck}{article}{
   author={Spruck, Joel},
   title={Uniqueness in a diffusion model of population biology},
   journal={Comm. Partial Differential Equations},
   volume={8},
   date={1983},
   number={15},
   pages={1605--1620},
   issn={0360-5302},
   review={\MR{729195}},
   doi={10.1080/03605308308820317},
}


\bib{W-cpde}{article}{
   author={Weiss, Georg S.},
   title={Partial regularity for weak solutions of an elliptic free boundary
   problem},
   journal={Comm. Partial Differential Equations},
   volume={23},
   date={1998},
   number={3-4},
   pages={439--455},
   issn={0360-5302},
   review={\MR{1620644}},
   doi={10.1080/03605309808821352},
}




\end{biblist}
\end{bibdiv}
\end{document}